\documentclass[10pt, reqno]{article}
\usepackage{geometry}

\usepackage[utf8]{inputenc}
\usepackage{amsmath}
\usepackage{amsthm}
\usepackage{amsfonts}
\usepackage{graphics}
\usepackage{amssymb}
\usepackage[utf8]{inputenc}
\usepackage{fancyhdr}

\def\phi{\varphi }
\def\arcosh{{\rm arcosh}\>}

\theoremstyle{plain}
\newtheorem{theorem}{Theorem}[section]

\newtheorem{lemma}[theorem]{Lemma}
\newtheorem{proposition}[theorem]{Proposition}
\theoremstyle{definition}
\newtheorem{definition}[theorem]{Definition}
\newtheorem{remark}[theorem]{Remark}

\theoremstyle{remark}

\numberwithin{equation}{section}

\title{Some limit theorems for random walks associated with hypergeometric functions of type BC}
\author{Merdan Artykov\footnote{This author has been supported by the Deutsche Forschungsgemeinschaft (DFG) via RTG 2131 \textit{High-dimensional Phenomena in Probability - Fluctuations and Discontinuity.}},  Michael Voit \\
	Fakult\"at Mathematik,  Technische Universit\"at Dortmund  \\
Vogelpothsweg 87, D-44221 Dortmund, Germany \\
	e-mail: merdan.artykov@tu-dortmund.de, michael.voit@math.tu-dortmund.de}

\date{\today}

\begin{document}
\maketitle

\begin{abstract}

The spherical functions of the
 noncompact Grassmann manifolds $G_{p,q}(\mathbb F)=G/K$ over 
${\mathbb F=\mathbb R}$, $\mathbb C, \mathbb H$ with rank $q\ge1$ and 
dimension parameter $p>q$ are
 Heckman-Opdam hypergeometric functions of type BC, when
the double coset spaces  $G//K$ are identified with the
  Weyl chamber $ C_q^B\subset \mathbb R^q$ of type B.
The associated double coset hypergroups on  $ C_q^B$ can be embedded into a continuous family
of commutative hypergroups $(C_q^B,*_p)$ with $p\in[2q-1,\infty[$ associated with these
hypergeometric functions by  R\"osler (2010). Several limit theorems for random walks 
on these  hypergroups were recently derived by Voit (2017).
 We here present further limit theorems when the
 time  as well as  $p$ tend to $\infty$. For integers $p$, 
this admits interpretations for group-invariant random walks on the Grassmannians $G/K$.

\end{abstract}




\section{Introduction}
In this paper we present several  limit theorems for group invariant random walks on the
 non-compact Grassmann manifolds $G_{p,q}(\mathbb F)=G/K$ over the (skew-)fields $\mathbb F=\mathbb R, \mathbb C, \mathbb H$.
 We state these results via
 the associated double coset spaces  $G//K$ 
which can be identified with  the Weyl chambers $C_q^B\subset \mathbb R^q$
of type $B$. The associated spherical functions, regarded as functions on  $C_q^B$, are then hypergeometric functions of type $BC$,
 and it turns out that the limit theorems can be derived for a  larger class of Markov chains on  $C_q^B$ 
whose transition probabilities are related these with  hypergeometric functions beyond the  group parameters.

Let us recapitulate some details of the general setting.
The Heckman-Opdam theory of hypergeometric functions associated 
with root systems  generalizes the  theory of spherical functions on
Riemannian symmetric spaces; see \cite{H},
\cite{HS} and \cite{O1} for the general theory,
and  \cite{R2}, \cite{RKV}, \cite{RV1}, \cite{S1}, \cite{S2}, \cite {Sch}, \cite{NPP} 
 for some recent developments.
In this paper we are mainly interested in the type $BC$, but we  also need
some facts on the $A$-case as a limit; see \cite{RKV}, \cite{RV1}.

We recapitulate that for the root system $A_{q-1}$, $q\ge2$, the  hypergeometric functions are 
connected with the groups $G:=GL(q,\mathbb F)$ with  maximal compact
subgroups $K:=U(q,\mathbb F)$. Moroever, for the root system $BC_{q}$, $q\ge1$,
 the  hypergeometric functions are related with
 the non-compact Grassmann manifolds $\mathcal G_{p,q}(\mathbb F):=G/K$
with $p> q$, where depending on $\mathbb F$, 
the group $G$
is one of the indefinite
orthogonal, unitary or symplectic  groups
$ SO_0(q,p),\, SU(q,p)$  or $Sp(q,p)$ with
 $K=
SO(q)\times SO(p), \, S(U(q)\times U(p))$ or $Sp(q)\times Sp(p),$
as maximal compact subgroup.

In all cases,  the $K$-spherical functions on $G$ 
(i.e., the nontrivial, $K$-biinvariant, multiplicative continuous functions 
on $G$) are nontrivial, multiplicative continuous functions on the double coset space $G//K$
 where $G//K$ carries  commutative 
double coset hypergroup structure.
The
$KAK$-decomposition of $G$ shows that $G//K$
may be identified with the Weyl chambers
$$C_q^A:= \{x=(x_1,\cdots,x_q)\in\mathbb  R^q: \> x_1\ge x_2\ge\cdots\ge
x_q\}$$
of type $A$
and
$$C_q^B:=\{x=(x_1,\cdots,x_q)\in\mathbb  R^q: \> x_1\ge x_2\ge\cdots\ge x_q\ge0\}$$
of type $B$ respectively. 
This identification is based on  a exponential
mapping $x\mapsto a_x\in G$ from the Weyl chamber to a system of
representatives $a_x$ of the double cosets in $G$  with
 \begin{equation}\label{a_t-A}a_x := e^{\underline x}  \end{equation}
 for $x\in C_q^A$ in
the $A$-case, 
and 
 \begin{equation}\label{a_t-BC}
a_x:=\begin{pmatrix} \cosh\underline x & \sinh \underline x& 0 \\
          \sinh  \underline x & \cosh  \underline x & 0 \\
   0 & 0 & I_{p-q}
       \end{pmatrix}  \end{equation}
for $x\in C_q^B$ in the $BC$-case with the diagonal matrices
$$ e^{\underline x}:= \text{diag}(e^{x_1},\ldots,e^{x_q}), \,
  \cosh\underline x = \text{diag}(\cosh x_1,
\ldots, \cosh x_q), \, \sinh \underline x = \text{diag}(\cosh x_1,
\ldots, \cosh x_q).$$
We  identify  $G//K$ with
$C_q^A$ or $C_q^B$ respectively. We also fix $q$ and, in the $BC$-case, $p>q$.

For the spherical functions we follow \cite{HS} and denote the
Heckman-Opdam hypergeometric functions associated with the root systems
\[  2\cdot A_{q-1} = \{ \pm 2(e_i-e_j): 1\leq i < j \leq
q\} \subset \mathbb R^q\]
 and
\[  2\cdot BC_q = \{ \pm 2e_i, \pm 4 e_i, \pm 2e_i \pm 2e_j: 1\leq i < j \leq
q\} \subset \mathbb R^q\]
by  $F_{A}(\lambda, k;t)$ and  $F_{BC}(\lambda, k;x)$ respectively
with spectral variable $\lambda \in \mathbb C^q$  and multiplicity parameter(s)
$k$. Here, $e_1,\ldots,e_q$ are the unit vectors in $\mathbb R^q$. 
The factor $2$ in both root systems comes from the known connections of
the Heckman-Opdam theory to spherical functions on symmetric spaces  in
\cite{HS} and references  there.
 In the $ A_{q-1} $-case, the spherical functions on $G//K\simeq C_q^A$ are
then 
\begin{equation}\label{def-spherical-a}
 \phi_\lambda^A(a_x):=\phi_\lambda^A(x):= e^{i\cdot \langle x-\pi(x),\lambda\rangle} \cdot F_{A}(i\pi(\lambda),d/2;\pi(x))
\quad\quad (x\in \mathbb R^q, \>\lambda \in \mathbb C^q  )
\end{equation}
with   multiplicity
$k=d/2$   where
 $$d:= \dim_{\mathbb R} \mathbb F\in\{1,2,4\} \quad\quad\text{for}\quad\quad
\mathbb F=\mathbb R, \mathbb C, \mathbb H,$$ 
 and where
$$\pi:\mathbb R^q\to \mathbb R^q_0:=\{t\in\mathbb R^q:\> x_1+\ldots+x_q=0\}$$
 is the orthogonal projection w.r.t.~the standard scalar product as in Eq.~(6.7) of \cite{RKV} and $a_t$ is identified with $x$.
 In the $BC$-case, the spherical functions on
 $G//K\simeq C_q^B$ are given by
\begin{equation}\label{def-spherical-bc} 
\phi_\lambda^p(a_x):=\phi_\lambda^p(x):=F_{BC}(i\lambda,k_p;x)
\quad\quad (x\in \mathbb R^q, \>\lambda \in \mathbb C^q)
\end{equation}
with  multiplicity
$$k_p=(d(p-q)/2, (d-1)/2, d/2)\subset \mathbb R^3$$
corresponding to the roots $\pm 2 e_i$,  $\pm 4 e_i$ and  $2(\pm  e_i\pm e_j)$ where again  $a_x$ is identified with $x$.

 In the $BC$-case, the
 associated double coset convolutions $*_{p}$ of measures on
$C_q^B$ are written down explicitly in  \cite{R2}
 for $p\ge 2q$  such 
that these convolutions and the associated product formulas for the associated
hypergeometric functions $F_{BC}$ above
  can be extended to  $p\in[ 2q-1,\infty[$ by
analytic continuation.
These convolutions  $*_{p}$ on the space $\mathcal{M}(C_q^B)$ of all bounded regular Borel measures on $C_q^B$
are  associative,
commutative, and  probability-preserving, and  they generate commutative
hypergroups $(C_q^B,*_{p})$ in the sense of Dunkl, Jewett, and
Spector with $0\in C_q^B$ as identity  by \cite{R2}. For  hypergroups we generally refer to 
  \cite{J} and \cite{BH}.
 The nontrivial multiplicative continuous functions of these commutative
 hypergroups $(C_q^B,*_{p})$ are precisely the functions $ \phi_\lambda^p$ with $\lambda\in
 \mathbb C^q$ by \cite{R2}. This means that
for all $x,y\in C_q^B$ and
 $\lambda\in \mathbb C^q$,
$$\phi_\lambda^p(x)\phi_\lambda^p(y)=
\int_{C_q^B}\phi_\lambda^p(t)\> d(\delta_x*_p\delta_y)(t)$$
where the probability measures $\delta_x*_p\delta_y\in \mathcal{M}^1(C_q^B)$ with
compact support are given by
\begin{equation}\label{convo-formel}
(\delta_x*_p\delta_y)(f)=
\frac{1}{\kappa_p}\int_{B_q}\int_{U(q,\mathbb F)} 
f\Bigr(\arcosh(\sigma_{sing}(\sinh\underline x \,w\,\sinh\underline y \,+\,
\>\cosh \underline x\> v\> \cosh \underline y
))\Bigr)
\> dv\> dm_p(w)
\end{equation}
for $f\in C(C_q^B)$.
Here, $dv$ means integration w.r.t.~the normalized Haar measure on
 $U(q,\mathbb F)$, $B_q$ is the matrix ball
$$B_q:= \{ w\in M_{q}(\mathbb F):\> w^*w\le I_q\},$$ 
and $dm_p(w)$ is the probability measure
\begin{equation}\label{probab-mp}
dm_p(w):=\frac{1}{\kappa_{p}} \Delta(I-w^*w)^{d(p/2+1/2-q)-1}\> dw
\quad \in \mathcal{M}^1( B_q)
\end{equation}
where $dw$ is the Lebesgue measure on the ball $B_q$, and the normalization
$\kappa_{p}>0$ is chosen such that $dm_p(w)$ is a
probability measure. For $p=2q-1$ there is a corresponding degenerated 
formula where $m_p \in \mathcal{M}^1( B_q)$  becomes singular; see
Section 3 of \cite{R1} for  details.

For fixed parameters  $p\in [2q-1,\infty[$ and $d=1,2,4$ we now consider random walks on the hypergroups
 $(C_q^B,*_{p})$ as follows: Fix a probability measure $\nu\in \mathcal{M}^1(C_q^B)$, and consider a time-homogeneous 
Markov process $(\tilde S_k^p)_{k\ge0}$ on $C_q^B$ with start at the hypergroup identity
 $0\in C_q^B$ and with the  transition
probability
$$P(\tilde S_{k+1}^p\in A|\> \tilde S_k^p=x)= (\delta_x *_p\nu)(A)
\quad\quad(x\in C_q^B, \> A\subset C_q^B \quad\text{a Borel set}).$$
Such Markov processes are called  random walks on the hypergroup $(C_q^B,*_{p})$ associated with the measure $\nu$.
Notice that we here  use $p$ as a superscript, as this $p$ may be variable below. The fixed parameters $q$ and $d$ are suppressed.

We shall present mainly two different types of CLTs for  $(\tilde S_k^p)_{k\ge0}$.

For the first type  in Section 5 we start with some probability measure $\nu$ having classical second moments.
For each constant $c\in[0,1]$ we consider the compression mapping $D_c(x):=cx$ on $C_q^B $ as 
well as the compressed probability measures $\nu_c:=D_c(\nu)\in \mathcal{M}^1(C_q^B)$ and 
 the associated random walks $(S^{(p,c)}_k)_{k\ge 0}$. We prove in Section 4
that $S^{(p,n^{-1/2})}_n$ converges for $n\to\infty$ in distribution to some 
``Gaussian'' measure $\gamma_{t_0}\in \mathcal{M}^1(C_q^B)$ 
which depends on $p$ where the time  $t_0\ge0$ can be computed via 
 second moment of $\nu$. Triangular CLTs of this type are well-known in 
probability theory on groups and hypergroups. We here in particular refer to \cite{BH} and references there 
for several results in this direction for
Sturm-Liouville hypergroups on $[0,\infty[$. Moreover, for integers $p\ge 2q$, this result is known 
for biinvariant random walks on noncompact Grassmannians; see e.g.  \cite{G1}, \cite{G2}, \cite{Te1}, \cite{Te2}, 
\cite{Ri}.

For the second CLT in Section 4 we study the random walks  $(\tilde S_k^p)_{k\ge0}$ for a given fixed 
probability measure $\nu\in \mathcal{M}^1(C_q^B)$ where the time $k$ as well
 as the dimension parameter $p$ tend to infinity in some coupled way.
It turns out that under suitable moment conditions on $\nu$ and for any
  sequence $(p_n)_n\subset [2q,\infty[$ with $p_n\to\infty$,
there are normalizing vectors $m(n)\in\mathbb R^q$ such that 
$(S^{p_n}_n-m(n))/\sqrt n$ tends in distribution to some classical $q$-dimensional normal distribution $N(0,\Sigma^2) $
where the norming vectors  $m(n)$ and the covariance matrix
$\Sigma^2$ are explicitly known and depend  $\nu$. For $q=1$, CLTs of this kind were given 
in \cite{Gr1} and \cite{V1} by completely different methods. Both proofs for $q=1$ however are based on the fact
that for $p\to\infty$, the hypergroup structures  $(C_1^B=[0,\infty[,*_{p})$ converge to some commutative semigroup structure on 
$C_1^B=[0,\infty[$ which is isomorphic with the additive semigroup $([0,\infty[,+)$.
 This observation finally shows that for large $p$,
$(S^{p_n}_n)_n$ behaves  like a sum of iid random variables which then leads to the CLT.
 For $q\ge2$, the situation is much more involved
as here for $p\to\infty$, the  hypergroup structures  $(C_q^B,*_{p})$ converge to the
 double coset structures $G//K$ in the case $A_{q-1}$
in some way, where  the dimension parameter $d=1,2,4$ remains unchanged; see \cite{RKV} and \cite{RV1} for the details.
 As for $q\ge2$, this limit structure is more complicated than for $q=1$,
 the details of the CLT and its proof in Section 3 will be more involved than  
in \cite{Gr1} and \cite{V1}. In fact, we will need stronger conditions either on the moments of $\nu$
 or on the rate of convergence of  $(p_n)_n$ to $\infty$ than in \cite{Gr1}; see Theorems
 \ref{growing parameters 1}, \ref{growing parameters 2} below. 
We remark that the CLTs in \cite{Gr1}, \cite{V1}, and here for the non-compact Grassmannians are related to other
CLTs for radial random walks on Euclidean spaces of large dimensions in \cite{Gr2} and references cited there.
We also point out that our CLTs for $p\to \infty$ are closely related to some CLT in
 the case $A_{q-1}$ in \cite{V2} which depends heavily 
on the concept of moment functions on commutative hypergroups; see \cite{BH} and \cite{Z1} for the general background. 
In fact, we shall need these moment functions for the $BC$-hypergroups 
 $(C_q^B,*_{p})$ as well as for the limit cases associated with
 the case $A_{q-1}$. These moment function will be essential to describe the norming vectors  $m(n)$ and the covariance matrix
$\Sigma^2$ above.
We shall collect several results on these functions in the next section. 
We point out that these results are mainly needed for the CLTs of Section 3, but not for those in Section 4.
We also remark that our CLTs for $p\to\infty$ are related to the research in \cite{B} on the
 limit behaviour of Brownian motions on hyperbolic spaces and noncompact Grassmannians when the dimension tends to infinity.

\section{Modified moments}

Generally, examples of moment functions on a commutative hypergroup can
 be obtained as partial derivatives of the multiplicative functions of the hypergroup w.r.t.~the spectral
 variables at the identity character; see \cite{BH}. To obtain explicit formulas for these moment functions
 for our particular examples on Weyl chambers, we start with explicit integral representations of the 
multiplicative functions in \cite{RV1} which are consequences of the well-known
Harish-Chandra integral representation of spherical functions.

We start with some notations from matrix analysis; we here usually refer to the monograph \cite{HJ}.
For  a Hermitian matrix
 $A= (a_{ij})_{i,j=1,\ldots,q}$ over $\mathbb F$ 
 we denote by $\Delta(A)$ the determinant of $A$, and by
$\Delta_r(A) = \det ((a_{ij})_{1\leq i,j\leq r})\,$ 
the $r$-th principal minor of $A$ for $r=1,\ldots,,q$.
 For $\mathbb F= \mathbb H, $
 these determinants are taken in the sense of Dieudonn\'{e}, i.e. 
$\det(A) = (\det_{\mathbb C} (A))^{1/2}$, when $A$ is considered as a complex
matrix. For each positive Hermitian $q\times q$-matrix $A$ and 
 $\lambda \in \mathbb C^q$ we consider the  power function
\begin{equation}\label{power-function}
 \Delta_\lambda(A) := \Delta_1(A)^{\lambda_1-\lambda_2} \cdot \ldots \cdot
\Delta_{q-1}(A)^{\lambda_{q-1}-\lambda_q}\cdot
\Delta_q(A)^{\lambda_{q}}.
\end{equation}
We shall also need the singular values $\sigma_1(a)\ge \sigma_2(a)\ge\ldots\ge \sigma_q(a)$ of a $q\times q$-matrix $a$
which are ordered by size and which are the ordered eigenvalues of $a^*a$.
Finally, for $x\in C_q^B$, $u\in U_q(\mathbb F)$, and $w\in B_q$, we define
\begin{equation}\label{def-g}
 g(x,u,w):=u^*(\cosh \underline x + \sinh\underline x \cdot w)(\cosh\underline x + \sinh\underline x \cdot w)^*u.
\end{equation}
We recapitulate the following facts; see Lemmas 4.10 and 4.8 of \cite{RV1}: 

\begin{lemma}\label{singular value estimate}
\begin{enumerate}
\item[\rm{(1)}] Consider the probability measures $m_p$ from (\ref{probab-mp}). Then
for each $n\in\mathbb{N}$  there exists a constant $C:=C(q,n,\mathbb{F})$ such that all $p\geqslant2q,$ 
\begin{equation}\label{scaling inequality}
\int_{B_q}\frac{\sigma_1(w)^{2n}}{\Delta(I-w^*w)^{2n}}dm_p(w)\leq \frac{C}{p^n}.
\end{equation}
\item[\rm{(2)}]  Let $x\in C^B_q, w\in B_q, u\in U(q,\mathbb{F})$ and $r=1,...,q$. 
 Then 
 $$\frac{\Delta_r(g(x,u,w))}{\Delta_r(g(x,u,0))} \in 
[(1-\tilde{x}\sigma_1(w))^{2r},(1+\tilde{x}\sigma_1(w))^{2r}] \quad\text{with}\quad \tilde{x}:=\min(x_1,1).$$
\end{enumerate}
\end{lemma}

We now recapitulate the moment functions in the $A$-case and then in $BC$-case from \cite{V2}.

\begin{definition}
The spherical functions of type A in (\ref{def-spherical-a}) satisfy
\begin{equation}\label{int-rep-a}
 \phi_\lambda^A(x) =\,
\int_{U(q,\mathbb F)} \Delta_{(i\lambda-\rho^A)/2}\bigl(u^{-1}e^{2\underline
  x}\,u\bigr) \> du  \quad\quad(x\in C_q^A)
\end{equation}
with the  half sum of positive roots
\begin{equation}\label{rho-a}
\rho^A:=(\rho^A_1,\ldots,\rho^A_q)\in  C_q^A \quad\quad\text{ with} \quad\quad
\rho^A_l:=\frac{d}{2}(q+1-2l) \quad\quad (l=1,\ldots,q);
\end{equation}
see  Section 3 of \cite{RV1}.
 Eq.~(\ref{int-rep-a}) in particular yields that  $\phi^A_{-i\rho^A}\equiv 1$,
 and that for $\lambda\in\mathbb R^n$ and $x\in C_q^A$, we have
$|\phi^A_{\lambda-i\rho^A}(x)|\le1$. 

We now follow \cite{V2}.
 For  multiindices $l=(l_1,\ldots,l_q)\in\mathbb N_0^q$ we define
the  moment functions
\begin{align}\label{moment-function-a}
m^A_l(x):=&
\frac{\partial^{|l|}}{\partial\lambda^l}\phi^A_{-i\rho^A-i\lambda}(x)
\Bigl|_{\lambda=0}:=\frac{\partial^{|l|}}{(\partial\lambda_1)^{l_1}\cdots(\partial\lambda_n)^{l_q}}
\phi^A_{-i\rho^A-i\lambda}(x)
\Bigl|_{\lambda=0}
\notag\\
=&\frac{1}{2^{|l|}}
\int_K (\ln\Delta_1(u^{-1}e^{2\underline x}\,u))^{l_1}\cdot
\left(\ln\left(\frac{\Delta_2(u^{-1}e^{2\underline x}\,u)}{\Delta_1(u^{-1}e^{2\underline x}\,u)}\right)\right)^{l_2}
\cdots
\left(\ln\left(\frac{\Delta_q(u^{-1}e^{2\underline x}\,u)}{\Delta_{q-1}(u^{-1}e^{2\underline t}\,u)}\right)\right)^{l_q}\> du
 \end{align}
of order $|l|:=l_1+\cdots+l_q$ for $t\in  C_q^A$. Notice that the
last equality in (\ref{moment-function-a}) follows  from  (\ref{int-rep-a})
 by interchanging integration and  derivatives. We denote the $j$-th unit vector by $e_j\in \mathbb{Z}^q_+ $ and the moment functions of order 1 and 2 by $m_{e_j}$ and $m_{e_j+e_k}\quad (j,k=1,..,q).$
The $q$ moment functions of first order lead to the vector-valued moment function
\begin{equation}\label{m1-vector}
m^A_{\bf 1}(x):=(m^A_{e_1}(x),\ldots,m^A_{e_q}(x))
\end{equation}
of first order.
Moreover, the  moment functions of second
order can be grouped by
\begin{align}\label{m2-matrix}
m^A_{\bf 2}(x):=&\left(\begin{array}{ccc} m^A_{2e_1}(x)&\cdots& m^A_{e_1+e_q}(x)\\
\vdots &&\vdots \\ m^A_{e_q+e_1}(x)&\cdots& m^A_{2e_q}(x) \end{array}\right)
\quad\quad\text{for}\quad x\in C_q^A.
\notag\end{align}
 We  now form the $q\times q$-matrices
$\Sigma^A(x):=m^A_{\bf 2}(x)-m^A_{\bf 1}(x)^x\cdot m^A_{\bf 1}(x)$. 
\end{definition}

These moment functions have the following basic properties; see Section 2 of \cite{V2}:

\begin{lemma}\label{prop-moment-function-a}
\begin{enumerate}
\item[\rm{(1)}]  There  is a constant
$C=C(q)$ such that for all $x\in C_q^A$, 
$\|m^A_{\bf 1}(x)-x\|\le C.$
\item[\rm{(2)}] For each $t\in C_q^A$, $\Sigma^A(x)$ is positive semidefinite.
\item[\rm{(3)}] For $x=c\cdot (1,\ldots,1)\in  C_q^A$ with $c\in\mathbb R$, 
 $\Sigma^A(x)=0$. For all other  $x\in  C_q^A$, $\Sigma^A(x)$ has rank $q-1$.
\item[\rm{(4)}]  All
second moment functions $m^A_{e_i+e_j}(x)$ are growing at most quadratically, and
  $m^A_{2e_1}(x)$ and  $m^A_{2e_q}(x)$ are in fact  growing quadratically.
  \item[\rm{(5)}]There exists a constant $C=C(p)$ such that for all $x\in C^A_q$  and $\lambda \in \mathbb{R}^q$, 
   $$ |\varphi^A_{-i\rho^A-\lambda}(x)-e^{i\langle \lambda, m^A_\mathbf{1}(x)\rangle}|\leq C||\lambda||^2.$$
\end{enumerate}\end{lemma} 

We now consider a probability measure $\nu\in \mathcal{M}^1(C^A_q)$. For $k\in \mathbb{N}$ we say that $\nu$ admits $k$-th moments of type A if for all $l\in \mathbb{N}^q_0 $ with $|l|\leq k$ the moment condition $m^A_l\in L^1(C^A_q,\nu)$ holds. \\
We then call  $m^A_l(\nu):=\int_{C^A_q}m^A_l(x)d\nu(x)$ the $l$-th multivariate moment of $\nu$.
The vector 
$$m^A_{\bf 1}(\nu):=  \int_{C_q^A} m_{\bf 1}(x)\> d\nu(x) \in C_q^A\subset \mathbb R^q$$ 
is called the dispersion of $\nu$. We also form  the modified symmetric $q\times q$-covariance matrix 
$$\Sigma^A(\nu):=  \int_G m_{\bf 2}\> d\nu \> 
-\> m^A_{\bf 1}(\nu)^t\cdot  m^A_{\bf 1}(\nu).$$ 

We are interested in the A-case only as a limit of the BC-case for $p\rightarrow \infty$.
 For this we need an additional transformation 
\begin{equation}\label{Trafo-T}
T: C_q^B\rightarrow C^B_q \subset C^A_q, \quad
 x=(x_1,...,x_q)\mapsto\ln \cosh x:=(\ln\cosh x_1,...,\ln \cosh x_q) 
\end{equation}
cf. \cite{RKV}, \cite{RV1}. We define the modified moment functions  $\tilde{m}_l(x):=m^A_l(T(x))$
which admit  modified integral representations similar to (\ref{moment-function-a}).
Moreover, for $\nu\in \mathcal{M}^1(C^B_q)$ we consider the
 image measure $T(\nu)\in \mathcal{M}^1(C^B_q)\subset \mathcal{M}^1(C^A_q)$.
 As $| x-\ln \cosh x|\leq \ln 2$ for all $x\in [0,\infty[$ by an elementary calculation, we see that
  for all multiindices $l$, the  $l$-th moment
 of type A of $\nu$ exists if and only if the $l$-th moment of type A  of  $T(\nu)$ exists.
 We put    $\tilde{m}_l({\nu}):=m^A_l(T(\nu))$ and $\tilde{\Sigma}(\nu):=\Sigma^A(T(\nu))$.

 We next turn to the $BC$-case.
\begin{definition}
For all  $p>2q-1$, $x\in C_q^B$, and $\lambda\in \mathbb C^q$, the  functions in (\ref{def-spherical-bc})
 satisfy
\begin{equation}\label{phi-int-kurz-bc}
\phi_{\lambda}^p(x)=\int_{B_q} \int_{U(q,\mathbb F)} 
 \Delta_{(i\lambda-\rho)/2}( g(x,u,w))
\> du\> dm_p(w)
\end{equation}
with the power function $\Delta_\lambda$ from (\ref{power-function}),
the half sum of positive roots
\begin{equation}\label{rho-BC}
 \rho =  \rho(p) =
 \sum_{i=1}^q \bigl( \frac{d}{2}(p+q+2-2i) -1 \bigr)e_i\,,
\end{equation}
 $g$ as above,  and  with $m_p(w)\in \mathcal{M}^1(B_q)$ from (\ref{probab-mp}); see \cite{RV1}.
As in \cite{RV1} we  define the  moment functions for $l=(l_1,\ldots,l_q)\in\mathbb N_0^q$ by:
\begin{align}\label{def-m1-bc}
&m_l^p(x):=
\frac{\partial^{|l|}}{\partial\lambda^l}\phi_{-i\rho^{BC}-i\lambda}^p(x)
\Bigl|_{\lambda=0}:=
\frac{\partial^{|l|}}{(\partial\lambda_1)^{l_1}\cdots(\partial\lambda_q)^{l_q}}
\phi_{-i\rho^{BC}-i\lambda}^p(x) \Bigl|_{\lambda=0}
\notag\\
=&
\frac{1}{2^{|l|}}\int_{B_q} \int_{U(q,\mathbb F)}
 (\ln\Delta_1( g(x,u,w)))^{l_1}\cdot
 \left(\ln \frac{\Delta_2( g(x,u,w))}{\Delta_1( g(x,u,w))}\right)^{l_2}
\cdots
 \left(\ln \frac{\Delta_q( g(x,u,w))}{\Delta_{q-1}( g(x,u,w))}\right)^{l_q}
\> du\> dm_p(w)
 \end{align}
for $x\in C_q^B$.
We also form  the vector-valued first moment function $m_{\bf 1}^p$, 
the matrix-valued second moment function $m_{\bf 2}^p$, as well as
$\Sigma^{p}(x):=m_{\bf 2}^p(x)-m_{\bf 1}^p(x)^t\cdot m_{\bf 1}^p(x)$
 as above.
\end{definition}

We  have the following basic properties; see Section 3 of \cite{V2}:

\begin{lemma}\label{moment-function-bc}
\begin{enumerate}
\item[\rm{(1)}]   There  is a constant
$C=C(p,q)$ such that for all $x\in C_q^B$, 
$$\|m_{\bf 1}^p(x)-x\|\le C.$$
\item[\rm{(2)}] For each $x\in C_q^B$, $\Sigma^{p}(x)$ is positive semidefinite.
\item[\rm{(3)}] $\Sigma^{p}(0)=0$, and for $x\in C_q^B\setminus\{0\}$, 
$\Sigma^p(x)$ has full rank $q$.
\item[\rm{(4)}]  All
second moment functions $m^p_{e_j+e_l}(x)$ are growing at most quadratically, and
  $m^p_{2e_1}$ is growing quadratically.
\item[\rm{(5)}]There exists a constant $C=C(p,q)$ such that for all $x\in C^B_q$  and $\lambda \in \mathbb{R}^q$, 
   $$ |\varphi^p_{-i\rho-\lambda}(x)-e^{i\langle \lambda, m^p_\mathbf{1}(x)\rangle}|\leq C||\lambda||_2^2.$$
\end{enumerate}\end{lemma} 

Similarly to the A-case, we also define  multivariate $l$-th moments,
 dispersions, and covariance matrices  of type BC($p$)
for measures $\nu\in \mathcal{M}^1(C^B_q)$.
\\

We next derive estimates for $|\tilde{m}_l(\nu)-m^p_l(\nu)|$ for $l\in \mathbb{N}^q_0$ and
 large $p$ under the assumption that these moments exist.
For this we first show that for a given $\nu\in \mathcal{M}^1(C^q_q)$ 
the existence of moments of some maximal order is independent
 from taking classical moments, moments of type A, or moments of type BC.
 For our purpose it will be sufficient to study the case with  $|l|$ even.\\
Let  $k\in \mathbb{N}_0$ and $\nu\in \mathcal{M}^1(C^q_q)$. We  say that  $\nu$
 admits finite  A-type moments of order at most $2k$ if
  $$\tilde{m}_{2k\cdot e_1},...,\tilde{m}_{2k\cdot e_q}\in L^1(C^B_q,\nu).$$
Indeed, it follows immediately from the definition of moment functions in (\ref{moment-function-a}) and H\"older's inequality, that in this case all moments of order at most $2k$ are $\nu$-integrable.
 Similarly, if $$m^p_{2k\cdot e_1},...,m^p_{2k\cdot e_q}\in L^1(C^B_q,\nu)$$ then we say that 
$\nu$ admits finite BC(p)-type moments of order at most $2k$.

\begin{proposition}\label{moment existence equivalence}
 For $k\in \mathbb{N}\text{ and }\nu \in \mathcal{M}^1(C^B_q)$ the following statements are equivalent:\\ 
(1)  $\nu$ admits all classical moments of order at most $2k$, i.e.  $\int _{C^B_q} x^{l_1}_1\cdots x^{l_q}_q d\nu(t)<\infty$ for all $l=(l_1,...,l_q)\in \mathbb{N}^q_0$ with $|l|\leq 2k$.\\
 (2) $\nu$ admits all moments of type A of order at most $2k$.\\ 
 (3) $T(\nu)$ admits all moments of type A of order at most $2k$.\\
 (4) For each $p\geq 2q-1$, $\nu$ admits all moments of type BC(p) of order at most $2k$.\\ 
 
  \end{proposition}

\begin{proof} 
 For (1)$\Rightarrow$(2) we prove that  $m^A_{2k\cdot e_1},...,m^A_{2k\cdot e_q}\in L^1(C^B_q,\nu)$. From (\ref{moment-function-a}) we have 
$$m^A_{2k\cdot e_j}(\nu)=\frac{1}{2^{2k}}\int_{C^B_q}\int_{U(q,\mathbb{F})}\left( \ln\Delta_{j+1}(u^*e^{ 2\underline{x}}u)-\ln\Delta_{j}(u^*e^{ 2\underline{x}}u)\right)^{2k} du\> d\nu(x).$$
We now recall from Lemma 4.2  \cite{V2}  that $jx_q\leq \ln\Delta_j(u^*e^{2\underline{x}}u)\leq jx_1$ for $u\in U(q,\mathbb{F}),x\in C^B_q$, and $j=1,...,q$. Therefore, from elementary inequalities  we obtain that 
\begin{equation}\label{tilde m finiteness}
m^A_{2k\cdot e_j}(\nu)\leq\frac{1}{2^{2k}}\int_{C^B_q}|(j (x_1-x_q)+x_q|^{2k}   \mathrm{d}\nu(x) <\infty.
\end{equation}
 To prove (2)$\Rightarrow$(1) it is sufficient to show that $\int_{C^B_q}x^{2k}_1\>d\nu(x)<\infty$. It can be easily seen that for every $u\in U(q,\mathbb{F})$ there exist coefficients  $c_i(u)\geq 0$ for $i=1,...q$ with $\sum^q_{i=1} c_i(u)=1$ such that $$\Delta_1(u^*e^{2\underline{x}}u)=\sum^q_{i=1} c_i(u) e^{2x_i}\geq c_1(u)e^{2x_1}.$$
Thus, as $2^{2k}(a^{2k}+b^{2k})\geq (a+b)^{2k}$ for 
$a=\ln(c_1(u)e^{2x_1})$ and $b=-\ln c_1(u)$,
\begin{align*}
\int_{U(q,\mathbb{F})}\int _{C^B_q}(\ln \Delta_1(u^*e^{2\underline{x}}u))^{2k}\> du\>d\nu(x)&\geq \int_{U(q,\mathbb{F})}\int _{C^B_q}(\ln (c_1(u)e^{2x_1}))^{2k}\> du\>d\nu(x) \\
&\geq -\int_{U(q,\mathbb{F})}(|\ln c_1(u)|)^{2k}\>du+\int _{C^B_q} x^{2k}_1\>d\nu(x).
\end{align*}

Now, Lemma 5.1 and Proposition 4.9 of \cite{V2} ensure that
 $\int_{U(q,\mathbb{F})}(|\ln c_1(u)|)^{2k}\>du$ is finite. Hence we have
 $\int_{C^B_q}x^{2k}_1 \>d\nu(x)<\infty$  as desired.\\
The equivalence of (2) and (3) follows from $$\frac{1}{4}u^*e^{2\underline{x}}u\leq u^*(\cosh \underline{x})^2u\leq \frac{1}{2}u^*e^{2\underline{x}}u$$ which implies that 
$$|\ln \Delta_j(u^*(\cosh \underline{x})^2u)-\ln \Delta_j(u^*e^{2\underline{x}}u)|\leq \ln 4.$$

 To prove (3)$\Rightarrow$ (4) we recall from Lemma 6.4 in \cite{V2} that 
  \begin{equation}\label{MV1 Lemma 6.4}
  |\ln \Delta_j g(x,u,w)-\ln\Delta_j(u^*(\cosh\underline{x})u)|\leq 2j\cdot \max(|\ln(1-\sigma_1(w))|,\ln(\sigma_1(w)+1)):=H_j(w).
  \end{equation} 
It can be easily seen that $\int_{B_q} \ln(1+\sigma_1(w))^{2k} \mathrm{d}m_p(w)$ is finite. 
Moreover, as $1\geq\sigma_1(w)\geq....\geq \sigma_q(w)\geq0 $ for $w\in B_q$ we have 
\begin{equation}\label{Sigma inequality}
\frac{1}{1-\sigma_1(w)}\leq \frac{2}{1-\sigma_1(w)^2}\leq 2\prod ^q_{r=1} \frac{1}{1-\sigma_r(w)^2}\leq \frac{2}{\Delta(I-w^*w)}.
\end{equation}
Now, from Lemma \ref{singular value estimate} and (\ref{Sigma inequality}) together with the elementary inequality 
\begin{equation}\label{elementary inequality}
|\ln(1+z)|\leqslant \frac{|z|}{1-|z|} \text{ for } |z|< 1
\end{equation}
 we get
  \begin{equation}\label{integral finiteness}
   \int_{B_q} |\ln(1-\sigma_1(w))|^{2k}\mathrm{d}m_p(w)\leq2^{2k}\int_{B_q} \sigma_1(w)^{2k}\cdot \Delta(I-w^*w)^{-2k}\mathrm{d}m_p(w)<\infty.
\end{equation}    
Hence,  $\int_{B_q}|H_j(q)|^{2k}dm_p(w)<\infty$ for $j=1,..,q$. 
   Therefore, using the elementary inequality $3^{2k}(a^{2k}+b^{2k}+c^{2k})\geq (a+b+c)^{2k}$  we have 
   \begin{align}\label{Trinagle inequality} m^p_{2k\cdot e_j}(\nu)\leq 
\left(\frac{3}{2}\right)^{2k}\int_{B_q\times U(q,\mathbb{F})\times C^B_q }&
\Bigl( |\ln \Delta_{j+1} g(x,u,w)-\ln\Delta_{j+1}(u^*(\cosh\underline{x})u)|^{2k} +\\ 
 &  + \left| \ln\Delta_{j+1}(u^*(\cosh \underline{x})u)-\ln\Delta_{j}(u^*(\cosh \underline{x})u)\right|^{2k}
 +\notag\\ 
&+|\ln \Delta_j g(x,u,w)-\ln\Delta_j(u^*(\cosh\underline{x})u)|^{2k}\Bigr)\mathrm{d}m_p(w)\>\mathrm{d}u
 \>\mathrm{d}\nu(x). 
\notag   \end{align}
If we use  (\ref{MV1 Lemma 6.4}), (\ref{integral finiteness}) and the assumption, we see that 
 the right hand side of  (\ref{Trinagle inequality}) is finite, which shows that $m^A_{2k\cdot e_j}(\nu)<\infty$.

Finally, the converse statement (4)$\Rightarrow$(3) follows analogously from 
 \begin{multline} m^A_{2k\cdot e_j}(\nu)\leq \left(\frac{3}{2}\right)^{2k}\int_{B_q\times U(q,\mathbb{F})\times C^B_q }[|\ln \Delta_{j+1}(u^*(\cosh\underline{x})u) -\ln\Delta_{j+1}g(x,u,w)|^{2k} \\ 
   \quad\quad\quad\quad+ |\ln\Delta_{j+1}g(x,u,w)-\ln\Delta_{j}g(x,u,w)|^{2k}\\
   +|\ln \Delta_j (u^*(\cosh\underline{x})u)-\ln\Delta_jg(x,u,w)|^{2k}]\mathrm{d}m_p(w)\mathrm{d}u \mathrm{d}\nu(x). 
   \end{multline}

\end{proof}
We now turn to the main result of the section:

\begin{proposition}\label{moment convergence Lemma}
Let $l=(l_1,....,l_q) \in \mathbb{N}^q_0$  with $|l|\geq 3$ and  $\nu \in \mathcal{M}(C^B_q)$. Assume that $\nu$ admits  finite moments of order  $4(|l|-2)$. Then, there exists a constant $C:=C(|l|,q,\nu)$ such that 

\begin{equation}
|\tilde{m}_l(\nu)-m^{p}_l(\nu)|\leqslant  \frac{C}{\sqrt{p}}. 
\end{equation}
 \end{proposition}

\begin{proof}
We consider the $|l|$ factors of the integrand in the integral representations
 (\ref{def-m1-bc}) of the moment functions $m_l^p$ and the modified version of (\ref{moment-function-a})
for  $\tilde{m}_l$. For $i=1,2,...,|l|$ these factors have the form: 
 \begin{align*}
 f_i(x,u,w):=\ln\Delta_r(g(x,u,w))-\ln \Delta_{r-1}(g(x,u,w)),\\
  \tilde{f}_i(x,u,w):=\ln\Delta_r(g(x,u,0))-\ln \Delta_{r-1}(g(x,u,0))
\end{align*}
with the convention $\Delta_0\equiv 1$ where $r\in \lbrace1,...,q\rbrace$ is the smallest integer with $i\leq l_1+...+l_r$.\\
Then, from Lemma \ref{singular value estimate}(2) and (\ref{elementary inequality})
 for all $i=1,...,|l|,x\in C^B_q, u\in U(q,\mathbb{F}), w\in B_q$ we obtain 
 \begin{align*}
 |f_i(x,u,w)-\tilde{f}_i(x,u,w)|&\leq 2\max_{r=1,...,q} |\ln\Delta_r(g(x,u,w))-\ln\Delta_r(g(x,u,0))|\\
 &\leqslant 4q\cdot\frac{\tilde{x} \sigma_1(w)}{1-\tilde{x}\sigma_1(w)}\leqslant 4q\tilde{x}\frac{\sigma_1(w)}{1-\sigma_1(w)}
  \end{align*}
where $\tilde{x}=\min \lbrace1,x\rbrace$.
Thus, by (\ref{Sigma inequality}),
 $$|f_i(x,u,w)-\tilde{f}_i(x,u,w)|\leq 8q\tilde{x}\frac{\sigma_1(w)}{\Delta(I-w^*w)}. $$
Now, notice that 
\begin{equation}
|\tilde{m}_l(\nu)-m^{p}_l(\nu)|=\left|\frac{1}{2^{|l|}}\int_{B_q\times U(q,\mathbb{F})\times C^B_q} \left(\prod^{|l|}_{i=1} f_i(x,u,w)-\prod^{|l|}_{i=1} \tilde{f}_i(x,u,w)\right) du  dm_p(w)d\nu(t)\right|
\end{equation}
Therefore, by a telescopic sum,
\begin{align}\label{telescopic sum}
|&\tilde{m}_l(\nu)-m^{p}_l(\nu)|=\notag\\
&=\Bigl|\frac{1}{2^{|l|}}\sum^{|l|}_{i=1}\int_{B_q\times U(q,\mathbb{F})\times C^B_q}\Bigl( (f_i(x,u,w)-\tilde{f}_i(x,u,w))\times
\notag\\ 
&\quad\quad\quad\quad\quad\quad\quad
\prod^{|l|}_{j=i+1}f_j(x,u,w)\prod^{i}_{k=1}\tilde{f}_k(x,u,w)\Bigr) du  dm_p(w)d\nu(x)\Bigr|\notag\\
&\leq\frac{1}{2^{|l|}}\sum^{|l|}_{i=1}\int_{B_q\times U(q,\mathbb{F})\times C^B_q}\Bigl| (f_i(x,u,w)-\tilde{f}_i(x,u,w))\times
\notag\\ 
&\quad\quad\quad\quad\quad\quad\quad
\prod^{|l|}_{j=i+1}f_j(x,u,w)\prod^{i}_{k=1}\tilde{f}_k(x,u,w)\Bigr| du  dm_p(w)d\nu(x)
\end{align}
We estimate the summands of the expression of the last formula of (\ref{telescopic sum}) in two ways:\\
Summands for $i=1$ and $|l|$:\\
>From Cauchy-Schwarz inequality, (\ref{telescopic sum}) and Lemma \ref{singular value estimate} we obtain that
\begin{multline}\label{Case 1}
\int_{B_q\times U(q,\mathbb{F})\times C^B_q}\left| (f_1(x,u,w)-\tilde{f}_1(x,u,w))\prod^{|l|}_{j=2}f_j(x,u,w) \right| du  dm_p(w)d\nu(x)\\
\leq \left(\int_{B_q\times U(q,\mathbb{F})\times C^B_q} |f_i(x,u,w)-\tilde{f}_i(x,u,w)|^2 du  dm_p(w)d\nu(t)\right)^{1/2}\times\\
\times \left(\int_{B_q\times U_0(q,\mathbb{F})\times C^B_q} \prod^{|l|}_{j=2}f_j(x,u,w) ^2 du  dm_p(w)d\nu(x)\right)^{1/2}\\
\leq M_1\cdot 8q \left(\int_{B_q} \frac{\sigma_1(w)^{2}}{\Delta(I-w^*w)^{2}}   dm_p(w)\right)^{1/2}
\leq M_1\cdot \frac{C}{\sqrt{p}}\quad\quad
\quad\quad\quad\quad\quad\quad\quad\quad\quad\quad\quad\quad\quad\quad\quad\quad\quad
\end{multline}
where $$M_1:=M_1(\nu,|l|,q)=8q\cdot\max_{r\in \mathbb{N}^q_0, |r|\leq 2(|l|-1)}\max\lbrace\tilde{m}_r(\nu),m^p_r(\nu)\rbrace $$
which is finite by initial assumption and Proposition \ref{moment existence equivalence}. Similarly, we obtain same upper bound for the $|l|$'s summand in (\ref{telescopic sum}).\\
  Now, let $i=2,...,q-1$. Here, we apply H\"older's inequality twice and obtain with the same arguments as above that
\begin{multline}\label{Case 2}
\left|\int_{B_q\times U_0(q,\mathbb{F})\times C^B_q}\left( (f_i(x,u,w)-\tilde{f}_i(x,u,w))\right.\prod^{|l|}_{j=i+1}f_j(x,u,w)\prod^{i-1}_{k=1}\left.\tilde{f}_k(x,u,w)\right) du  dm_p(w)d\nu(x)\right|\\
 \leq\left(\int_{B_q\times U_0(q,\mathbb{F})\times C^B_q}|(f_i(x,u,w)-\tilde{f}_i(x,u,w)|^2du  dm_p(w)d\nu(t)\right)^{1/2}\\
\times\left(\int_{B_q\times U_0(q,\mathbb{F})\times C^B_q}\prod^{|l|}_{j=i+1}|f_j(x,u,w)|^4du  dm_p(w)d\nu(x)\right)^{1/4} \\
\times\left(\int_{B_q\times U_0(q,\mathbb{F})}\prod^{i-1}_{k=1}|\tilde{f}_k(x,u,w)|^4du  dm_p(w)d\nu(x)\right)^{1/4}\quad\quad\\
 \leq M_2\cdot \frac{C}{\sqrt{p}}\quad\quad\quad\quad\quad\quad\quad\quad\quad\quad\quad\quad
\quad\quad\quad\quad\quad\quad\quad\quad\quad\quad\quad\quad\quad\quad\quad
\end{multline}
where 
$$M_2:=M_2(\nu,|l|,q)=8q\cdot\max_{r\in \mathbb{N}^q_0, |r|\leq 4(|l|-2)}\max\lbrace\tilde{m}_r(\nu),m^p_r(\nu)\rbrace $$
which is again finite by our assumption and Proposition \ref{moment existence equivalence}.
Thus, the estimates  (\ref{Case 1}) and (\ref{Case 2}) give the desired assertion.
\end{proof}

\section{Spherical Fourier transform}
In this section we collect some well-known methods and facts about the  spherical Fourier transform of type A and BC. We start with the identification of all multiplicative functions and of the dual space in accordance with \cite{R2} and \cite{NPP} for $p\geq2q-1$ in the BC-case. \\ 
 The set of all continuous multiplicative functions
 $$\chi(C_q^B,*_p):=\lbrace f:C^B_q\rightarrow \mathbb{C}: f \text{ continuous}, \int_{C^B_q}f d(\delta_x*_p\delta_y)=f(x)f(y)  \rbrace $$
  is given by $\lbrace\varphi^p_\lambda:\lambda\in \mathbb{C}^q \rbrace$. Moreover,
 the set $\chi_b(C^B_q,*_p)$ of bounded functions in $\chi(C_q^B,*_p)$ is equal to
 $\lbrace \varphi^p_\lambda: \> \Im\lambda\in co(W_q\cdot\rho) \rbrace$ where $co$ denotes the convex hull, and
$W^B_q$ the Weyl group of type $B_q$ acting on $\mathbb{C}^q$. The dual space $$(C^B_q,*_p)^\wedge:= \lbrace f\in \chi_b(C^B_q,*_p), f(x^-)=\overline{f(x)}\rbrace$$ 
is  $\lbrace \varphi^p_\lambda:\lambda\in C^B_q \textit{ or } \lambda \in i\cdot co(W^B_q\cdot\rho)\rbrace$. Finally, the support of Plancherel measure is the set  $\lbrace\varphi^p_\lambda:\lambda\in C^B_q\rbrace$.
\begin{definition}
Let $\nu\in \mathcal{M}^1(C^B_q)$. The BC-type spherical (or hypergroup) Fourier transform is given by $$\mathcal{F}^p_{BC}(\nu)(\lambda):=\int _{C^B_q} \varphi^p_\lambda(x)d\nu(x)$$
for $\lambda\in \lbrace \lambda\in \mathbb{C}^q :\Im\lambda\in co(W^B_q\cdot \rho)\rbrace$.
\end{definition}
 We now give some estimates on spherical functions and Fourier transforms from $\cite{V2}$.

\begin{lemma}\label{wohldefiniertheitslemma} For all $x\in C^B_q$, $\lambda\in \mathbb{R}^q$, and $l\in \mathbb{N}^q_0,$
$$\left|\frac{\partial^{|l|}}{\partial\lambda^l}\varphi^p_{\lambda-i\rho}(x)\right|\leqslant m^p_l(x)$$
\end{lemma}

\begin{lemma}\label{Fourier ableitung lemma} 
Let $k\in \mathbb{N}_0$  and assume that $\nu \in \mathcal{M}^1(C^B_q)$ admits finite  $k$-th modified moments.
 Then, for all $\lambda\in \mathbb{C}^q$ with $\Im \lambda\in co(W^B_q\cdot\rho )$, $\mathcal{F}^p_{BC}(\nu)(\cdot)$ is $k$-times continuously differentiable, and for all $l\in \mathbb{N}^n_0$ with $|l|\leqslant k,$
\begin{equation} \label{Fourier ableitung} 
\frac{\partial ^{|l|}}{\partial \lambda^l}\mathcal{F}^p_{BC}(\nu)(\lambda)=\int _{C^B_q}\frac{\partial^{|l|}}{\partial \lambda^l} \varphi^p_{\lambda}(x)d \nu(x).
\end{equation}
In particular, 
\begin{equation} \label{Fourier ableitung at 0} \frac{\partial ^{|l|}}{\partial \lambda^l}\mathcal{F}_{BC}(\nu)(-i\rho)=\int _{C^B_q}m^p_l(x)d \nu(x).
\end{equation}
\end{lemma}

\begin{remark} There are corresponding results to the Lemmas \ref{wohldefiniertheitslemma}
 and \ref{Fourier ableitung lemma} for the A-case with the corresponding moment functions $  m^A_l$
 for $l\in \mathbb{N}^q_0$ and the Fourier transform $\mathcal{F}_A$ and
 $\nu \in \mathcal{M}^1(C^A_q)$; see Lemmas 6.1, 6.2 in \cite{V2}.
 \end{remark}

\section{Limit theorems  for  growing parameters with outer normalization } \label{Growing parameters}
In this section we derive two types of limit theorems  for random walks when the time and the dimension parameter $p$ tend to infinity. The statements of both limit theorems  are similar, but the assumptions on the moments and the relation between the 
 time parameter $n$ and and dimension parameter $p$ are different. We first present a CLT where we assume some restriction on $(p_n)_{n\geq1}$:

 \begin{theorem}\label{growing parameters 1}
 Let $(p_n)_{n\geq 1}\subset]2q-1,\infty[$ be  an increasing  sequence with  $\lim_{n \rightarrow \infty}n/ p_n=0$.
 Let $\nu\in \mathcal{M}^1(C^B_q)$  with $\nu\neq \delta_0$ and   with second moments.
 Consider the associated random walks  $(S^{p}_n)_{n\geqslant 0 }$  on $C^B_q$ for $p\geq 2q-1$. Then 
  $$\frac{S^{p_n}_n-n\cdot\tilde{m}_\mathbf{1}(\nu)}{\sqrt{n}}$$ 
converges in distribution to $ \mathcal{N}(0,\tilde{\Sigma}(\nu)).$  
   \end{theorem}

 \begin{proof}
 We know  from Lemma 4.2(2) of \cite{RV1} that  there exists a constant  $C>0$ such that
for  all $p>2q-1, x\in C^B_q, \lambda\in \mathbb{R}^q$,
  $$|\varphi^p_{\lambda-i\rho}(x)-\varphi^A_{\lambda-i\rho^A}(\ln \cosh x)|\leqslant
 C\cdot \frac{ \lVert \lambda \rVert_1 \cdot \tilde{x}}{p^{1/2}} $$
 where $\lVert\lambda\rVert_1:=|\lambda_1|+\dots\ |\lambda_q|$   and $\tilde{x}:=min(x_1,1)\geqslant 0$.
Hence, denoting  the half sums of positive roots of type BC associated 
with  $p_n$ as described in (\ref{rho-BC}) by  $\rho(n):=\rho^{BC}(p_n)$, for all $\nu\in  \mathcal{ M}^1(C^B_q)$, we get 
 \begin{equation} \label{integral von differenz}
 \left|\int _{C^B_q} \varphi^{p_n}_{\lambda-i\rho(n)} (x)d\nu(x)-\int _{C^B_q} \varphi^A _{ \lambda-i\rho^A}(\ln \cosh x)d\nu(x)\right |\leqslant C\cdot\frac{\lVert \lambda\rVert_1}{\sqrt{p_n}}.
 \end{equation}

Let $\nu^{(n,p)}\in \mathcal{M}^1(C^B_q)$ be the law of $S^p_n$. Then, $T(S^{p_n}_n)$ has the distribution $T(\nu^{(n,p_n)}) $  whose A-type spherical Fourier transform satisfies
\begin{equation} \label{Foueier A}
\mathcal{F}_A (T(\nu^{(n,p_n)}))(\lambda-i\rho^A)=\int _{C^A_q} \varphi^A_{\lambda-i\rho^A}(x)d T(\nu^{(n,p_n)})(x)=\int _{C^B_q} \varphi^A _{\lambda-i\rho^A}(\ln \cosh x)d\nu^{(n,p_n)}(x)
\end{equation}
for $\lambda\in \mathbb R^q$.  Furthermore, by plugging $\nu^{(n,p_n)}$ into (\ref{integral von differenz})  we get 
 \begin{align}\label{Fourier transition proof 1} 
 \mathcal {F}_A (T(\nu^{(n,p_n)}))(\lambda-i\rho^A)&=\int_{C^B_q} \varphi^{p_n}_{\lambda-i\rho(n)}d\nu^{(n,p_n)}(x)+O(\frac{\lVert \lambda  \rVert_1}{ p^{1/2}_n})\notag\\
 &=\mathcal{F}_{BC}^{p_n}(\nu^{(n,p_n)})(\lambda-\rho(n))+O(\frac{\lVert \lambda  \rVert_1}{ p^{1/2}_n})\notag\\
&=\left(\mathcal{F}_{BC}^{p_n}(\nu)(\lambda-\rho(n))\right)^n+O(\frac{\lVert \lambda  \rVert_1}{ p^{1/2}_n})\notag\\
 &=\left(\int_{C^B_q} \varphi^A_{\lambda-i\rho^A}(\ln \mathrm{cosh} x)d\nu (x) \right)^n+O(\frac{\lVert \lambda  \rVert_1}{ p^{1/2}_n})\notag\\
 &=\left(\mathcal{F}_A(T(\nu ))(\lambda-i\rho^A)+O(\frac{\lVert \lambda  \rVert_1}{ p^{1/2}_n})\right)^n+O(\frac{\lVert \lambda  \rVert_1}{ p^{1/2}_n}).
 \end{align}
Using the the initial moment assumption and Lemma \ref{moment existence equivalence}  we see that the first and second modified moments $\tilde{m}_\mathbf{1}$ and $\tilde{m}_\mathbf{2}$ exist. Moreover, all entries of the modified covariance matrix $$\tilde{\Sigma}(\nu)=\tilde{m}_\mathbf{2}(\nu)-\tilde{m}_\mathbf{1}(\nu)^t\cdot \tilde{m}_\mathbf{1}(\nu)$$ 
 are finite. \\
  By Lemma \ref{Fourier ableitung lemma}, the Taylor expansion of   $\mathcal{F}_A(T(\nu))(\lambda-i\rho^A)$ for $|\lambda|\rightarrow 0$ is given by
 \begin{equation}\label{Taylor exp. type A}
\mathcal{F}_A(T(\nu))(\lambda-i\rho^A)=1-i \langle\lambda,\tilde{m}_\mathbf{1}(\nu)\rangle-\lambda\tilde{m}_\mathbf{2}(\nu)\lambda^t+o(|\lambda|^2). 
\end{equation} 
  Using  the initial assumption that $O(1/\sqrt{n p_n})=o(1/n)$  we obtain
  \begin{multline*}
 E(\varphi^A_{\lambda/\sqrt{n}-i\rho^A}(T(S^{p_n}_n)) e^{i\langle\lambda,\sqrt{n} \tilde{m}_\mathbf{1}(\nu)\rangle}=\mathcal{F}_A(T(\nu^{(n,p_n)}))(\lambda/\sqrt{n}-i\rho^A)\cdot e^{i\langle\lambda,\sqrt{n} \tilde{m}_\mathbf{1}(\nu)\rangle}\quad \quad \quad \quad\quad\\
  =\left[\left(\mathcal{F}_A(T(\nu ))(\frac{\lambda}{\sqrt{n}}-i\rho^A)+O(\frac{\lVert \lambda  \rVert_1}{ \sqrt{np_n}})\right)^n+O(\frac{\lVert \lambda  \rVert_1}{ \sqrt{np_n}})\right]\cdot  e^{i\langle\lambda,\frac{ \tilde{m}_\mathbf{1}(\nu)}{\sqrt{n}}\rangle n}\\
  =\left[\left(1-\frac{ i \langle\lambda,\tilde{m}_\mathbf{1}(\nu)\rangle}{\sqrt{n}}-\frac{\lambda\tilde{m}_\mathbf{2}(\nu)\lambda^t}{2n}+o(\frac{1}{n})\right)\right.\times\quad\quad\quad\quad\quad\quad\\
\quad\quad\quad\quad\quad\quad\quad\quad\times\left.\left(1+ \frac{ i \langle\lambda,\tilde{m}_\mathbf{1}(\nu)\rangle}{\sqrt{n}}- \frac{  \langle\lambda,\tilde{m}_\mathbf{2}(\nu)\rangle^2}{2n}+o(\frac{1}{n})\right)\right]^n\\
  =\left(1- \frac{\lambda\tilde{\Sigma}(\nu)\lambda^t}{2n} +o(\frac{1}{n})\right)^n.\quad\quad\quad\quad\quad\quad\quad\quad\quad\quad\quad\quad\quad\quad\quad\quad\quad\quad\quad\quad
  \end{multline*}
Thus, 
  \begin{equation}\label{expression 1}
\lim_ {n\rightarrow \infty}E(\varphi^A_{\lambda/\sqrt{n}-i\rho^A}(T(S^{p_n}_n))\cdot\exp(i\langle\lambda, \tilde{m}_\mathbf{1}(\nu)\rangle\sqrt{n}))=\exp(-\lambda\tilde{\Sigma}(\nu)\lambda^t/2).
  \end{equation}
  On the other hand, from Lemma \ref{prop-moment-function-a}(5) we have 
  \begin{equation} \label{expression2}\lim_{n\rightarrow \infty}E(\varphi^A_{\lambda/\sqrt{n}-i\rho^A}(T(S^{p_n}_n))-\exp(-i\langle\lambda, \tilde{m}_\mathbf{1}(S^{p_n}_n)\rangle/\sqrt{n}))=0.
   \end{equation}
(\ref{expression 1})  and (\ref{expression2}) and the fact that
 $|e^{i\langle\lambda,\sqrt{n}\tilde{m}_\mathbf{1}(\nu)\rangle}|\leqslant 1$ together yield 
 that for all $\lambda \in \mathbb{R}^q$, $$ \lim_{n\rightarrow \infty}\exp(-i\langle\lambda, (\tilde{m}_\mathbf{1}(S^{p_n}_n)-n\cdot \tilde{m}_\mathbf{1}(\nu)\rangle)/\sqrt{n})=\exp(-\lambda\tilde{\Sigma}(\nu)\lambda^t/2).$$
   L\'evy's continuity theorem for the classical q-dimensional Fourier transform implies
 that\\ $(\tilde{m}_\mathbf{1}(S^{p_n}_n)-n\cdot \tilde{m}_\mathbf{1}(\nu)\rangle)/\sqrt{n}$
 tends to the normal distribution $\mathcal{N}(0,
  \tilde{\Sigma}(\nu))$.\\
 Hence, by  Lemma \ref{prop-moment-function-a}(1), the definition of $T$, and by
 $\lim_{x\rightarrow\infty}(x-\ln\cosh x)=\ln 2$, we obtain that
 $(S^{p_n}_n-n\tilde{m}_\mathbf{1}(\nu))/\sqrt{n}\rightarrow\mathcal N(0,\tilde{\Sigma}(\nu))$ as claimed.
 \end{proof}

For the weak LLN we only need  first moments of $\nu \in \mathcal{M}^1(C^B_q)$: 

\begin{theorem}\label{growing parameters 1 LLN}
 Let $(p_n)_{n\geq 1}\subset]2q-1,\infty[$ be  an increasing  sequence with $\lim_{n \rightarrow \infty}n/ p_n=0$.\\
Moreover, let  $\nu\in \mathcal{M}^1(C^B_q)$ be with $\nu\neq \delta_0$ and first moments.
 Consider the associated random walks  $(\tilde S^{p}_n)_{n\geqslant 0 }$  on $C^B_q$ for $p> 2q-1$ and let $\varepsilon>\frac{1}{2}.$ Then 
  $$\frac{1}{n^\varepsilon}(\tilde S^{p_n}_n-n\cdot\tilde{m}_\mathbf{1}(\nu)) \longrightarrow 0 \text{ in probability.}$$ 
In particular,
 $\frac{\tilde S^{p_n}_n}{n}\longrightarrow   \tilde{m}_\mathbf{1}(\nu)$ in probability. 
 \end{theorem}

\begin{proof}
The proof is very similar to that of Theorem \ref{growing parameters 1}. In fact,
 (\ref{Fourier transition proof 1}), (\ref{Taylor exp. type A}),  $\varepsilon>\frac{1}{2}$ and \\$O(1/\sqrt{n p_n})=o(1/n)$  show that 
 \begin{align}
 E(\varphi^A_{\frac{\lambda}{n^{\varepsilon}}-i\rho^A}(T&(\tilde S^{p_n}_n)) e^{i\langle\lambda,n^{1-\varepsilon} \tilde{m}_\mathbf{1}(\nu)\rangle}
=\mathcal{F}_A(T(\nu^{(n,p_n)}))(\lambda/n^\varepsilon-i\rho^A)\cdot e^{i\langle\lambda,n^{1-\varepsilon} \tilde{m}_\mathbf{1}(\nu)\rangle}\notag\\ 
  &=\left[\left(\mathcal{F}_A(T(\nu ))(\frac{\lambda}{n^\varepsilon} -i\rho^A)+O(\frac{\lVert \lambda  \rVert_1}{ n^\varepsilon\sqrt{p_n}})\right)^n+O(\frac{\lVert \lambda  \rVert_1}{ \sqrt{np_n}})\right]\cdot  e^{i\langle\lambda,\frac{\tilde{m}_\mathbf{1}(\nu)}{n^{\varepsilon}} \rangle n}\notag\\
  & =\left[\left(1-\frac{ i \langle\lambda,\tilde{m}_\mathbf{1}(\nu)\rangle}{n^{\varepsilon}}+O(\frac{\lVert \lambda  \rVert_1}{n^{\varepsilon+1/2}})\right)\right.\left.\left(1+ \frac{ i \langle\lambda,\tilde{m}_\mathbf{1}(\nu)\rangle}{n^\varepsilon} +O(\frac{\lVert \lambda  \rVert_1}{n^{2\epsilon}})\right)\right]^n\notag\\
  &=\left(1+o(\frac{\|\lambda\|^2}{n})\right)^n.\notag
\end{align}
Thus,
 \begin{equation}\label{expressionLLn1}
 \lim_{n\rightarrow\infty} E(\varphi^A_{\frac{\lambda}{n^{\varepsilon}}-i\rho^A}(T( \tilde S^{p_n}_n)) e^{i\langle\lambda,n^{1-\varepsilon} \tilde{m}_\mathbf{1}(\nu)\rangle}=1
 \end{equation}
for all $\lambda\in \mathbb{R}^q$.
On the other hand, from Lemma \ref{prop-moment-function-a}(5) we have 
  \begin{equation} \label{expressionLLn2}\lim_{n\rightarrow \infty}E(\varphi^A_{\lambda/n^{\varepsilon}-i\rho^A}(T(\tilde S^{p_n}_n))-\exp(-i\langle\lambda, \tilde{m}_\mathbf{1}(\tilde S^{p_n}_n)\rangle/n^{\varepsilon}))=0.
   \end{equation}
 (\ref{expressionLLn1}), (\ref{expressionLLn2}), and 
 $|e^{i\langle\lambda,\sqrt{n}\tilde{m}_\mathbf{1}(\nu)\rangle}|\leqslant 1$ yield 
 that for all $\lambda \in \mathbb{R}^q$,
  $$ \lim_{n\rightarrow \infty}\exp(-i\langle\lambda, (\tilde{m}_\mathbf{1}(\tilde S^{p_n}_n)-
n\cdot \tilde{m}_\mathbf{1}(\nu)\rangle)/n^\varepsilon)=1.$$
 The classical  L\'evy's continuity theorem implies
 that $(\tilde{m}_\mathbf{1}(\tilde S^{p_n}_n)-n\cdot \tilde{m}_\mathbf{1}(\nu))/n^{\varepsilon} \longrightarrow 0$
 in distribution and hence in probability.
 The proof can be now completed  as that of Theorem \ref{growing parameters 1}. 
\end{proof}

\begin{remark}
For rank $q=1$ the CLT  \ref{growing parameters 1}  was derived in \cite{Gr1} 
with different techniques under weaker assumptions, namely without the restriction $n/p_n\rightarrow 0$.
 The proof in \cite{Gr1} relies on  the convergence of the moment functions 
\begin{equation}\label{moment convergence }
(m^{p}_1(x))^2-m^{p}_2(x)\rightarrow 0 
\end{equation}
on $[0,\infty[$ for $p\rightarrow \infty$. However, for $q\geq2$ this convergence is  no longer available.
\end{remark}

We next try to get rid of the restriction  $n/p_n\rightarrow 0$.
For this we assume fourth moments.

\begin{theorem}\label{growing parameters 2}
   Let $(p_n)_{n\geq 1}$ be  an increasing  sequence with $p_1\geqslant 2q-1$ and 
 $\lim_{n \rightarrow \infty} p_n=\infty$. Moreover,  let $\nu\in \mathcal{M}^1(C^B_q)$ 
 with $\nu\neq \delta_0$ and with  fourth moments. 
 Consider the  associated random walks  $(S^{p}_n)_{n\geqslant 0 }$  on $C^B_q$ for $p\geq 2q-1$. 
 Then
 $$\frac{S^{p_n}_n-n\cdot m^{p_n}_\mathbf{1}(\nu)}{\sqrt{n}}$$ converges in distribution to $ \mathcal{N}(0,\tilde{\Sigma}(\nu))$.
\end{theorem}

\begin{proof}
We first notice that by Taylor's theorem and Proposition \ref{moment convergence Lemma}  
for all $p\geqslant 2q-1$,
\begin{align}\label{Fourirer BC estimate }
\left|E(\varphi^{p}_{\lambda/\sqrt{n}-i\rho}(S^{p}_n))-
\left(1-\frac{ i \langle\lambda,m^{p}_\mathbf{1}(\nu)\rangle}{\sqrt{n}}-\frac{\lambda m^{p}_\mathbf{2}(\nu)\lambda^t}{2n}\right)\right|&\leq \sum_{l\in \mathbb{N}^q,|l|=3} m^p_l(\nu) \frac{\lambda^{l_1}_1...\lambda^{l_q}_q}{l_1!...l_q!} \nonumber\\
&\leq\frac{1}{n^{3/2}}\sum_{l\in \mathbb{N}^q,|l|=3} (\tilde{m}_l(\nu)+C/\sqrt{p}) \frac{\lambda^{l_1}_1...\lambda^{l_q}_q}{l_1!...l_q!} \nonumber\\
&\leq K_1\frac{\lVert \lambda\rVert^3_\infty}{n^{3/2}}
\end{align}
for some constant $K_1>0$ which is independent of $p$.
 Analogously, for all $p\geqslant 2q-1,$
\begin{equation}\label{exponential BC estimate}
\left|e^{i\langle\lambda,\sqrt{n} m^{p}_\mathbf{1}(\nu)\rangle}-\left(1+\frac{ i \langle\lambda,m^{p}_\mathbf{1}(\nu)\rangle}{\sqrt{n}}-\frac{  \langle\lambda,m^{p}_\mathbf{1}(\nu)\rangle^2}{2n}\right)\right|\leqslant K_2\frac{\lVert \lambda\rVert^3_\infty}{n^{3/2}}
\end{equation}
for some $K_2>0$ independent of $p$.\\
Using estimates (\ref{Fourirer BC estimate }) and (\ref{exponential BC estimate}) we now follow 
similar paths as in the proof of  Theorem \ref{growing parameters 1}.
 We however use the BC-type Fourier transform and BC-moments instead of objects of  type  $A$,
 and then approximate  $A$-type moments by $BC$-type moments using Proposition \ref{moment convergence Lemma}. Now, we have
\begin{multline*}
 E(\varphi^{p_n}_{\lambda/\sqrt{n}-i\rho(n)}(S^{p_n}_n)) e^{i\langle\lambda,\sqrt{n} m^{p_n}_\mathbf{1}(\nu)\rangle}=\mathcal{F}^{p_n}_{BC}(\nu^{(n,p_n)})(\lambda/\sqrt{n}-i\rho(n))\cdot e^{i\langle\lambda,\sqrt{n} m^{p_n}_\mathbf{1}(\nu)\rangle}\\
  \quad\quad\quad\quad\quad\quad\quad\quad\quad\quad\quad=\left[\left(1-\frac{ i \langle\lambda,m^{p_n}_\mathbf{1}(\nu)\rangle}{\sqrt{n}}-\frac{\lambda m^{p_n}_\mathbf{2}(\nu)\lambda^t}{2n}+o(\frac{1}{n})\right)\right.\times\\
   \quad\quad\quad\quad\quad\quad\quad\quad\quad\quad\quad\quad\quad\quad\quad\quad\quad\quad\times\left.\left(1+\frac{ i \langle\lambda,m^{p_n}_\mathbf{1}(\nu)\rangle}{\sqrt{n}}- \frac{  \langle\lambda,m^{p_n}_\mathbf{1}(\nu)\rangle^2}{2n}+o(\frac{1}{n})\right)\right]^n\\
  =\left(1- \frac{\lambda{\Sigma}^{p_n}(\nu)\lambda^t}{2n} +o(\frac{1}{n})\right)^n\quad\quad\quad\quad\quad\quad\quad\quad\quad\quad\quad\quad
  \end{multline*}
  From  Lemma \ref{moment convergence Lemma} we also obtain that $$|\lambda\Sigma^{p_n}(\nu)\lambda^t-\lambda\tilde{\Sigma}(\nu)\lambda^t|=O(\frac{|\lambda|^2}{\sqrt{p_n}})$$ for $p_n\rightarrow \infty$.
  Therefore, we have \begin{align*}
  \lim_{n\rightarrow\infty}E(\varphi^{p_n}_{\lambda/\sqrt{n}-i\rho(n)}(S^{p_n}_n)) e^{i\langle\lambda,\sqrt{n} m^{p_n}_\mathbf{1}(\nu)\rangle}&=\lim_{n\rightarrow\infty}\left(1- \frac{\lambda\tilde{\Sigma}(\nu)\lambda^t}{2n}+\frac{\lambda({\Sigma}^{p_n}(\nu)-\tilde{\Sigma(\nu))}\lambda^t}{2n} +o(\frac{1}{n})\right)^n\\
   &=\exp(-\lambda\tilde{\Sigma}(\nu)\lambda^t/2)
\end{align*}   
  On the other hand from the Lemma \ref{moment-function-bc}(5) we have 
    \begin{equation}\lim_{n\rightarrow \infty}E(\varphi^{p_n}_{\lambda/\sqrt{n}-i\rho(n)}(S^{p_n}_n)-\exp(-i\langle\lambda, {m}^{p_n}_\mathbf{1}(S^{p_n}_n)\rangle/\sqrt{n}))=0.
   \end{equation}
The rest of the proof is now  analogous to that of Theorem \ref{growing parameters 1}.
\end{proof}

We next consider a weak LLN whenever second moments  exist:

\begin{theorem}\label{growing parameters 2 LLN}
Let $(p_n)_{n\geq 1}\subset ]2q-1 ,\infty[$ be  increasing   with 
 $\lim_{n \rightarrow \infty} p_n=\infty$. Let $\nu\in \mathcal{M}^1(C^B_q)$ 
 with $\nu\neq \delta_0$ and with  second moments. 
 Consider the  associated random walks  $(\tilde S^{p}_n)_{n\geqslant 0 }$  on $C^B_q$ for $p> 2q-1$. Let $\varepsilon> \frac{1}{2}$. Then 
 $$\frac{1}{n^\varepsilon}(\tilde S^{p_n}_n-n \cdot m^{p_n}_\mathbf{1}(\nu))\longrightarrow 0 \text{ in probability. }$$

\end{theorem}

\begin{proof}
As in the proof of the preceding  theorem we have for 
 $p>2q-1$
\begin{equation}\label{Fourirer BC estimate LLn2}
\Bigl|E(\varphi^{p}_{\lambda/n^\varepsilon-i\rho(n)}(\tilde S^{p}_n))-
\Bigl(1-\frac{ i \langle\lambda,m^{p}_\mathbf{1}(\nu)\rangle}{n^\varepsilon}\Bigr)\Bigr|
leq K_1\frac{\lVert \lambda\rVert^3_\infty}{n^{2\varepsilon}}
\end{equation}
for some $K_1>0$ independent of $p$. Moreover, in the same way,
\begin{equation}\label{exponential BC estimate LLN2}
\left|e^{i\langle\lambda,n^{\varepsilon}\cdot m^{p}_\mathbf{1}(\nu)\rangle}-\left(1+\frac{ i \langle\lambda,m^{p}_\mathbf{1}(\nu)\rangle}{n^{\varepsilon}}\right)\right|\leqslant K_2\frac{\lVert \lambda\rVert^3_\infty}{n^{2\varepsilon}}.
\end{equation}
Using  (\ref{Fourirer BC estimate LLn2}) and (\ref{exponential BC estimate LLN2}) we now 
follow the proof of  Theorem \ref{growing parameters 2}.
  For $\lambda\in \mathbb{R}^q$ we have
\begin{multline*}
 E(\varphi^{p_n}_{\lambda/n^\varepsilon-i\rho(n)}(\tilde S^{p_n}_n)) e^{i\langle\lambda,n^\varepsilon\cdot m^{p_n}_\mathbf{1}(\nu)\rangle}=\mathcal{F}^{p_n}_{BC}(\nu^{(n,p_n)})(\lambda/n^\varepsilon-i\rho(n))\cdot e^{i\langle\lambda,n^\varepsilon\cdot m^{p_n}_\mathbf{1}(\nu)\rangle}\\
  \quad  \quad  \quad  \quad  \quad  \quad \quad  \quad  \quad\quad\quad\quad\quad\quad\quad\quad=\left[\left(1-\frac{ i \langle\lambda,m^{p_n}_\mathbf{1}(\nu)\rangle}{n^\varepsilon}+o(\frac{1}{n})\right)\right.\left.\left(1+\frac{ i \langle\lambda,m^{p_n}_\mathbf{1}(\nu)\rangle}{n^\varepsilon}+o(\frac{1}{n})\right)\right]^n\\
  =\left(1+o(\frac{1}{n})\right)^n.\quad\quad\quad\quad\quad\quad\quad\quad\quad\quad\quad\quad\quad\quad\quad\quad\quad\quad
  \end{multline*}
  Therefore, for $\lambda\in\mathbb{R}^q$,
 $ \lim_{n\rightarrow\infty} E(\varphi^{p_n}_{\lambda/n^\varepsilon-i\rho(n)}(\tilde S^{p_n}_n)) e^{i\langle\lambda,n^\varepsilon\cdot m^{p_n}_\mathbf{1}(\nu)\rangle}=1$.
 
  On the other hand from the Lemma \ref{moment-function-bc}(5) for all $\lambda\in \mathbb{R}^q$ we have 
    \begin{equation}\lim_{n\rightarrow \infty}E(\varphi^{p_n}_{\lambda/n^\varepsilon-i\rho(n)}(\tilde S^{p_n}_n)-\exp(-i\langle\lambda, {m}^{p_n}_\mathbf{1}(\tilde S^{p_n}_n)\rangle/n^\varepsilon))=0.
   \end{equation}
 Hence, by L\'evy's continuity theorem,
 $$(\tilde{m}_\mathbf{1}(\tilde S^{p_n}_n)-n\cdot m^{p_n}_\mathbf{1}(\nu))/n^\varepsilon \longrightarrow 0 \text{ in distribution.}$$
 As in the proof of Theorem \ref{growing parameters 1}, this readily implies the claim.
\end{proof}

\section{A central limit theorem with inner normalization }
In this section we present some CLT for fixed  $p$ in  the following setting: Fix some nontrivial probability measure
 $\nu\in \mathcal{M}^1(C^B_q)$ with some moment condition and for $d\in ]0,1]$ consider
 the component-wise compression map $D_d: x\mapsto d\cdot x$ on $C^B_q$  as well as
 compressed measure $\nu_d:=D_d(\nu)\in \mathcal{M}^1(C^B_q)$.  For given $\nu$ and $d$
 we consider the  random walk $(S^{(p,d)}_n)_{n\geqslant 0}$ associated with $\nu_d$. 
 We  investigate the limiting behavior 
 of $(S^{(p,n^{-1/2})}_n)_{n\geqslant 1}$. This case can be seen as CLT with inner standardization in contrast
 to the case with $(S^p_n)_{n\geq 0}$ in Section 3 where we consider CLT with outer standardization $n^{1/2}$.
 These two CLTs exhibit different limiting procedures. The limit theorem for  $(S^{(p,n^{-1/2})}_n)_{n\geqslant 1}$ 
in the  rank 1 case was studied  by Zeuner \cite{Z1}. In the group cases, this CLT is related with the CLTs in  \cite{G1}, \cite{G2}, \cite{Te1}, \cite{Te2}, 
\cite{Ri}.

\begin{definition}\label{gamma t} Let $p\ge 2q-1$ and  $t\geq 0$. A probability measure
 $\gamma_t=\gamma_t(p)\in \mathcal{M}^1(C^B_q)$ is  called \textit{BC($p$)-Gaussian} with time parameter $t$
and shape parameter $p$ if $$\mathcal{F}^p_{BC}(\gamma_t)(\lambda)=exp{(\frac{-t(\lambda^2_1+...+ \lambda^2_q +\|\rho\|^2_2)}{2})}$$
 for all $\lambda\in C^B_q \cup i\cdot co(W^B_q\cdot\rho)\subset\mathbb{C}^q$. 
\end{definition} 

We notice that by injectivity of the hypergroup Fourier transform (see \cite{J}), the
 measures $\gamma_t$ are determined uniquely,
 and that they form a weakly continuous convolution semigroup $(\gamma_t)_{t\geq 0}$, i.e. for all $s,t\geq 0$ 
$\gamma_s*_p \gamma_t=\gamma_{s+t}$ and  $\gamma_0=\delta_0$. 
The existence of the measures $\gamma_t$ for $t> 0$ is not obvious at the beginning, but we shall
 see from the proof of he following CLT that the $\gamma_t$ exist.

\begin{theorem}\label{inner standartisations} 
Let $\nu\in \mathcal{M}^1(C^B_q)$ with $\nu\neq \delta_0$ and with finite second  moments. 
Let  $$t_0:=\frac{2}{qd}\int_{C^B_q}\|x\|_2^2d\nu(x).$$ Then, 
 $(S^{(p,n^{-1/2})}_n)_{n\geq 1}$ tends in distribution for $n\rightarrow \infty$ to $\gamma_{\frac{t_0}{p+1}}$.
 \end{theorem}

For the proof we need some information on $\varphi^p_\lambda$:

  \begin{lemma} \label{partial derivatives of phi}
  Let $p\in [2q-1,\infty[$ be fixed.  Then:
\begin{enumerate}
\item[\rm{(1)}]  For all $i,j=1,2,...,q$ with $i\neq j$ and all $\lambda\in \mathbb{C}^q,$
\begin{equation}
 \frac {\partial}{\partial  x_i}\varphi^p_\lambda(0)=0 \text{     and     }  \frac{\partial^2}{\partial  x_i \partial x_j}\varphi^p_\lambda(0)=0 
\end{equation}  
\item[\rm{(2)}]  For all  $i=1,2,...,q$, and $\lambda \in C^B_q \cup i\cdot co(W_q\cdot\rho) $,
 $$\frac {\partial^2}{\partial  x^2_i}\varphi^p_\lambda(0)=-\frac{2(\lambda^2_1+...+ \lambda^2_q +\|\rho\|^2_2)}{(p+1)qd}<0.$$  
\end{enumerate}
  \end{lemma}

  \begin{proof}  The  functions $\varphi^p_\lambda(x)$ are invariant under the action of
the  Weyl group of  type BC w.r.t.~$x$. Therefore, $\varphi^p_\lambda(x_1,.., x_q)$ is even in each $x_i,$ which leads to (1). Moreover, as  $\varphi^p_\lambda(x_1,....,x_q)$ is invariant under  permutations,
$\frac {\partial^2}{\partial  x^2_i}\varphi^p_\lambda(0)$   is independent of $i$.
To complete the proof of (2), we recall from Eq. (1.2.6) in  \cite{HS} that for all $\lambda \in \mathbb{C}^q$  the
 function  $F_{BC}(\lambda,k_p,\cdot)$ is the unique  solution to the eigenvalue problem 
 \begin{equation}
Lf=-(\lambda^2_1+...+ \lambda^2_q +\|\rho\|^2_2) f
\end{equation}  
  for $x\in int (C^B_q)=\lbrace x\in C^B_q :x_1>x_2>...>x_q>0\rbrace$ with $f(0)=1$ with the differential operator
\begin{multline}
L:=\sum_{1\leq i\leq q} \left[ \frac{\partial^2_i}{\partial x^2_i}+( k_1 \coth (x_i)+2k_2 \coth (2x_i))\frac{\partial_i}{\partial x_i} \right]\\
+k_3\sum_{1\leq i< j\leq q}\left[ \coth (x_i+x_j)\left( \frac{\partial_i}{\partial x_i}+ \frac{\partial_j}{\partial x_j}\right)+\coth (x_i-x_j)\left( \frac{\partial_i}{\partial x_i}- \frac{\partial_j}{\partial x_j}\right)\right].
\end{multline}
Now, using  part (1),  $\varphi^p_\lambda(x)=F_{BC}(i\lambda,k_p,x)$, and the 
Taylor expansion of $\coth$  around $0$, we have 
\begin{align*}
-(\lambda^2_1+...+ \lambda^2_q +\|\rho\|^2_2)\varphi^p_\lambda(0)&= \lim_{\|x\|\rightarrow 0} L\varphi^p_\lambda(x)\\
&= (q +qk_1+2qk_2+q(q-1)k_3)\left.\frac{\partial^2_1}{\partial x^2_1}\varphi^p_\lambda(x)\right|_{x=0}\\
&=\frac{(p+1)qd}{2}\cdot\left.\frac{\partial^2_1}{\partial x^2_1}\varphi^p_\lambda(x)\right|_{x=0}
\end{align*}
for all $\lambda\in \mathbb{C}^q$. Finally, as $co(W^B_q\cdot \rho)$ is contained in  $\lbrace x\in\mathbb{R}^q:\|x\|_2\leq \|\rho\|_2 \rbrace$, the final statement of (2) is also clear.
\end{proof}

 \begin{proof}[Proof of Theorem \ref{inner standartisations}]
  Lemma \ref{partial derivatives of phi}  and  $\varphi^p_\lambda(x)\leq 1$ 
for  $x\in C^B_q$ ensure that
 there exists $c> 0$ with  $$1-c(x^2_1+x^2_2+...+x^2_q)\leqslant \varphi^p_\lambda(x)\text{ for all } x\in C^B_q.$$

 Consequently by Taylor expansion, 
 $$n\left|\varphi^p_\lambda(\frac{x}{\sqrt{n}})-1+\frac{\lambda^2_1+...+ \lambda^2_q +\|\rho\|^2_2}{(p+1)qd} \cdot \frac{\|x\|_2^2}{n} \right|
\le C\|x\|_2^2$$ for some constant $C>0$ where $\|x\|_2^2$ is integrable w.r.t $\nu$ by our 
 assumption. Thus, by dominated convergence,
   $$\lim_{n\rightarrow \infty }n\int_{C^B_q}\left(\varphi^p_\lambda(\frac{x}{\sqrt{n}})-1+\frac{(\lambda^2_1+...+ \lambda^2_q +\|\rho\|^2_2)}{(p+1)qd} \cdot \frac{\|x\|_2^2}{n}\right)d\nu(x)=0.$$ 
  Rewriting this relation as 
  $$ \int_{C^B_q}\varphi^p_\lambda(\frac{x}{\sqrt{n}})d\nu(x)=1-\frac{1}{n}\frac{(\lambda^2_1+...+ \lambda^2_q +\|\rho\|^2_2)}{(p+1)qd} \cdot \int_{C^B_q}\|x\|_2^2d\nu(x)+ o(\frac{1}{n})$$
  we obtain
\begin{align*}
\mathcal{F}^p_{BC}(\mathbb{P}_{S^{(p,n^{-1/2})}_n})(\lambda)&=\int_{C^B_q}\varphi^p_{\lambda}(\frac{x}{\sqrt{n}})d\nu^{(n)}(x)=\left[\int_{C^B_q}\varphi^p_{\lambda}(\frac{x}{\sqrt{n}})d\nu(x)\right]^n\\
&=\left( 1-\frac{1}{n}\cdot\frac{(\lambda^2_1+...+ \lambda^2_q +\|\rho\|^2_2)}{(p+1)qd}  \int_{C^B_q}\|x\|_2^2d\nu(x)+ o(\frac{1}{n})\right)^n
\end{align*}
and
\begin{align*}
\lim_{n\rightarrow \infty}\mathcal{F}^p_{BC}(\mathbb{P}_{{S^{(p,n^{-1/2})}_n}})(\lambda)&=\exp\left(-\frac{(\lambda^2_1+...+ \lambda^2_q +\|\rho\|^2_2)}{(p+1)qd} \cdot  \int_{C^B_q}\|x\|_2^2d\nu(x)\right)\\&=\exp\left(-\frac{t_0(\lambda^2_1+...+ \lambda^2_q +\|\rho\|^2_2)}{2(p+1)}   \right)
\end{align*}
for all $\lambda \in\mathbb{R}^q\cup i\cdot co(W^B_q\cdot \rho)$. Hence, by L\'evy's continuity theorem 
 on  commutative hypergroups (Theorem 4.2.4(iv) in \cite{BH}) there exists a bounded positive measure in $\nu\in \mathcal{M}^+_b(C^B_q)$ with 
\begin{equation}\label{vague limit}
\mathcal{F}^p_{BC}(\nu)(\lambda)=\exp\left(-\frac{t_0(\lambda^2_1+...+ \lambda^2_q +\|\rho\|^2_2)}{2(p+1)} \right)
\end{equation}
for all $\lambda\in \mathbb{R}^q$, and $(\mathbb{P}_{S^{n{-1/2}}_n})_{n\geq 1}$ converges to $\nu$ weakly.\\
 Since we have  $\mathcal{F}^p_{BC}(\nu)(-i\rho)=1$, the limiting positive measure $\nu$ is indeed a probability measure. This implies that $(\mathbb{P}_{S^{(p,n^{-1/2})}_n})_{n\geq 1}$ converges weakly to $\nu=\gamma_{\frac{t_0}{p+1}}$  as desired.
\end{proof}
\begin{remark}
The considerations in the above proof yield that the probability measures $\gamma_t$ in Definition $\ref{gamma t}$
 above indeed exist.
\end{remark}

\section{A law of large numbers for inner normalizations and growing parameters }
We here present a further limit theorem for $(S_n^{(p,n^{-1/2})})_{n\geq1}$ for $p,n\to\infty$.
It will turn out that  the limit
is a point measure, i.e., we obtain a weak law of large numbers:

\begin{theorem}\label{compresssed growing case}
Let $\nu\in \mathcal{M}^1(C^B_q)$  with $\nu\neq \delta_0$ and finite second  moments. 
Let $t_0:=\frac{2}{qd}\int_{C^B_q}\|x\|_2^2d\nu(x)$ be as in Theorem \ref{inner standartisations}  and  $(p_n)_{n\geq1}\subset[ 2q-1,\infty[$ 
be  increasing with  $\lim_{n\rightarrow \infty}n/p_n=0$. Then, 
    $(S^{(p_n,n^{-1/2})}_n)_{n\geq 1}$ tends in probability for $n\rightarrow \infty$ to the constant
 $$\ln\left(e^{t_0/4}+\sqrt{e^{t_0/2}-1}\right)\cdot (1,\ldots,1).$$
\end{theorem}
For the proof  we first recapitulate the Taylor expansion for $\varphi^A_\lambda(x)$ at $x=0$ from \cite{G1}, where it was obtained for $d=1$. The expansion for $d=2,4$ follows similarly.

\begin{lemma}\label{Case-A taylor expansion}
For $\|x\|_2\rightarrow 0 $,
$$\varphi^A_{\lambda}(x)=
1+\frac{1}{qd}(\lambda_1+\lambda_2+...+\lambda_q)\sum_{k=1}^q x_k+R_\lambda(x)$$
with  $R_\lambda(x)=\sum_\alpha f_\alpha(\lambda)P_\alpha(x) $
where the $P_\alpha(x)$ are symmetric polynomials in $x_1,...,x_q$ which are homogeneous of order $\ge 2$.
\end{lemma} 

We also need the following fact:

\begin{lemma}\label{rho inclusion}
For $p\geq 2q-1$, the half sum $\rho=\rho^{BC}(p)$ satisfies the condition  $\rho^A-\rho\in co(W^B_q\cdot \rho)$, where $W^B_q$ is the Weyl group of type $B_q$.
\end{lemma}
\begin{proof}
Denote  $\hat{\rho}:=(\rho_q,\rho_{q-1}...,\rho_1)$. Then, obviously $,-\rho, -\hat{\rho}\in W^B_q\cdot\rho$. On the other hand we have  \begin{align*}
\rho^A-\rho =\left( \frac{d}{2}(p+1)-1\right)(1,....,1)=\frac{1}{2}(-\rho-\hat{\rho}).
\end{align*}
This proves the result. 
\end{proof}
\begin{proposition}\label{phi BC taylor}
Let $\nu$, $t_0$   and $(p_n)_{n\geq 1}$ be defined as in Theorem \ref{compresssed growing case}.
 Consider the half sum of positive roots $\rho(n):=\rho^{BC}(p_n)$ of type BC associated 
with the parameters $p_n$ as described in
(\ref{rho-BC}).
 Then, for all $\lambda\in \mathbb C^q$ with  $\Im \lambda=\rho^A$,
\begin{equation}
\int_{C^B_q}\varphi^{p_n}_{\lambda-i\rho(n)}(\frac{x}{\sqrt{n}})d\nu(x)=1+\frac{t_0}{4n}\cdot 
\sum_{k=1}^q (\lambda_k-i\rho^A_k)+o(1/n) \text{ as } n\rightarrow \infty.
\end{equation}
\end{proposition}

\begin{proof}
Lemma \ref{Case-A taylor expansion} and the Taylor expansion $\ln \cosh x=x^2/2+O(x^4)$ show that for all $\lambda\in \mathbb{C}^q$ with such that $\Im\lambda\in co(W^A_q\cdot\rho^A)$
 \begin{equation}\label{Acase modififed} 
\varphi^A_\lambda(\ln\cosh \frac{x}{\sqrt{n}})=1+\sum_{i=1}^q\lambda_i\frac{\|x\|^2_2}{2ndq}+R_\lambda(\frac{\|x\|^2}{n}) 
\end{equation}
for $n\rightarrow\infty$. 
On the other hand, Theorem 4.2(2) in \cite{RV1} states that 
\begin{equation}\label{RV1 citation}
|\varphi^p_{\lambda-i\rho(n)}(\frac{x}{\sqrt{n}})-
\varphi^A_{\lambda-i\rho_A}(\ln \cosh\frac{x}{\sqrt{n}})|\leq C\cdot\frac{\|\lambda\|_1\cdot\min(1, x_1/\sqrt n)}{\sqrt{p}}
\end{equation}
for all $\lambda\in \mathbb{C}^q$ such that $\Im\lambda-\rho(n) \in co(W^B_q\cdot \rho(n)).$
Notice that the analysis of the proof of Theorem 4.2(2) in \cite{RV1} shows that (\ref{RV1 citation}) is in fact  precisely valid for $$\lambda\in \lbrace \lambda\in \mathbb{C}^q:\Im\lambda-\rho(n)\in co(W^B_q\cdot\rho(n)) \text{ and } \Im\lambda-\rho^A\in co(W^A_q\cdot\rho^A)\rbrace.$$
If we combine (\ref{Acase modififed}) and (\ref{RV1 citation}) and use the Lemma \ref{rho inclusion} we see that as $p_n/n\rightarrow \infty$ 
\begin{equation}
\left|\varphi^{p_n}_{\lambda-i\rho(n)}(\frac{x}{\sqrt{n}})-1-\sum_{k=1}^q(\lambda_k-i\rho^A_k)\frac{\|x\|^2_2}{2ndq}\right|=o(\frac{\|x\|^2_2}{n}) \quad \text{for all } \lambda\in\mathbb{C}^q \text{ with } \Im\lambda=\rho^A 
\end{equation}
which, by integrating w.r.t $\nu$ yields the result.

\end{proof}
\begin{proof}[Proof of the Theorem \ref{compresssed growing case}]
Let $\nu^{(n,p_n)}$ be the $n$-fold $*_{p_n}$ convolution power of $\nu$. The  Proposition \ref{phi BC taylor} shows that for all $\lambda\in \mathbb{C}^q$ with $\Im \lambda=\rho^A$
\begin{align*}
\lim_{n\rightarrow\infty} \int_{C^B_q}\varphi^{p_n}_{\lambda-i\rho(n)}(\frac{x}{\sqrt{n}})d\nu^{(n,p_n)}(x)=&\lim_{n\rightarrow\infty} \left(\int_{C^B_q}\varphi^{p_n}_{\lambda-i\rho(n)}(\frac{x}{\sqrt{n}})d\nu(x)\right)^n\\
=&\lim_{n\rightarrow \infty}\left( 1+\frac{t_0}{4n}\cdot \sum_{k=1}^q (\lambda_k-i\rho^A_k)+o(1/n)\right)^n\\
=& e^{\frac{t_0}{4} \cdot \sum_{k=1}^q (\lambda_k-i\rho^A_k)}.
\end{align*}
Thus, using (\ref{RV1 citation}) we have that 
\begin{align*}
\lim_{n\rightarrow\infty}\mathcal{F}^A(\mathbb{P}_{T(S^{(p_n,n^{-1/2})}_n)})(\lambda-i\rho^A)&= \lim_{n\rightarrow \infty}\int_{C^B_q}\varphi^A_{\lambda-i\rho^A}(\ln \cosh\frac{x}{\sqrt{n}})d\nu^{(n,p_n)}(x)\\
&=\lim_{n\rightarrow\infty} \int_{C^B_q}\varphi^{p_n}_{\lambda-i\rho(n)}(\frac{x}{\sqrt{n}})d\nu^{(n,p_n)}(x)\\&=e^{\frac{t_0}{4} \cdot \sum_{k=1}^q (\lambda_k-i\rho^A_k)}
\end{align*}
for all $\lambda\in\mathbb{C}^q$ with $\Im\lambda=\rho^A.$ By making substitution $\lambda\mapsto\lambda+i\rho^A$ above, we get
\begin{equation}\label{Limit condition}
\lim_{n\rightarrow\infty}\mathcal{F}^A(\mathbb{P}_{T(S^{(p_n,n^{-1/2})}_n)})(\lambda)=e^{\frac{t_0}{4} \cdot \sum_{k=1}^q \lambda_k}
\end{equation}
 for all $\lambda\in \mathbb{R}^q$.
On the other hand from (\ref{int-rep-a}) we see that 
\begin{equation}
e^{\frac{t_0}{4}\cdot \sum_{k=1}^q \lambda_k}= \varphi^A_{\lambda}(\frac{t_0}{4}(1,...,1))
=\mathcal{F}^A(\delta_{\frac{t_0}{4}(1,....,1)})(\lambda)
\end{equation}
for  $\lambda\in\mathbb{C}^q$ with $\Im\lambda\in co(W^A_q\cdot\rho^A).$
Since, (\ref{Limit condition}) holds on $\mathbb{R}^q$, i.e., on the support of the Plancherel measure, 
 the L\'evy continuity theorem  for commutative hypergroups (see Theorem 4.2.11 in \cite{BH})
 yields  that  $\mathbb{P}_{T(S^{(p_n,n^{-1/2})}_n)}$ converges vaguely to $\delta_{\frac{t_0}{4}(1,...,1)}.$ Moreover, as
the $\mathbb{P}_{T(S^{(p_n,n^{-1/2})}_n)}$ and $\delta_{\frac{t_0}{4}(1,...,1)}$ are  probability measures, the sequence 
$(\mathbb{P}_{T(S^{(p_n,n^{-1/2})}_n)})_{n}$ is tight and  the convergence becomes weak. Since $T^{-1}$ is continuous,
 the continuous mapping theorem shows that $\mathbb{P}_{S^{(p_n,n^{-1/2})}_n}$ converges weakly to 
$T^{-1}(\delta_{\frac{t_0}{4}\cdot(e_1,...,e_q)})=\delta_{\ln\left(e^{t_0/4}+\sqrt{e^{t_0/2}-1}\right)\cdot (1,...,1)}. $
This completes the proof.
\end{proof}

\end{document}